\documentclass[12pt]{amsart}
\usepackage{amscd,amsmath,amsthm,amssymb}
\usepackage{pstcol,pst-plot,pst-3d}
\usepackage{lmodern,pst-node}
\usepackage{multicol}
\usepackage{epic,eepic}
\usepackage{lineno}
\usepackage{amsfonts,amssymb,amscd,amsmath,enumerate,verbatim}
\psset{unit=0.7cm,linewidth=0.8pt,arrowsize=2.5pt 4}

\newpsstyle{fatline}{linewidth=1.5pt}
\newpsstyle{fyp}{fillstyle=solid,fillcolor=verylight}
\definecolor{verylight}{gray}{0.97}
\definecolor{light}{gray}{0.9}
\definecolor{medium}{gray}{0.85}
\definecolor{dark}{gray}{0.6}
 \def\frk{\frak}               

 \def\mm{{\frk m}}
 
 \def\F{{\mathcal F}}

 \def\opn#1#2{\def#1{\operatorname{#2}}} 
 \opn\chara{char} \opn\length{\ell} \opn\pd{pd} \opn\rk{rk}
 \opn\projdim{proj\,dim} \opn\injdim{inj\,dim} \opn\rank{rank}
 \opn\depth{depth} \opn\grade{grade} \opn\height{height}
 \opn\embdim{emb\,dim} \opn\codim{codim}
 
 \opn\Tr{Tr} \opn\bigrank{big\,rank}
 \opn\superheight{superheight}\opn\lcm{lcm}
 \opn\trdeg{tr\,deg}
 \opn\reg{reg} \opn\lreg{lreg} \opn\ini{in} \opn\lpd{lpd}
 \opn\size{size} \opn\sdepth{sdepth}
 \opn\link{link}\opn\fdepth{fdepth}\opn\lex{lex}
 \opn\div{div} \opn\Div{Div} \opn\cl{cl} \opn\Cl{Cl}
 \opn\Spec{Spec} \opn\Supp{Supp} \opn\supp{supp} \opn\Sing{Sing}
 \opn\Ass{Ass} \opn\Min{Min}\opn\Mon{Mon}
 \opn\Ann{Ann} \opn\Rad{Rad} \opn\Soc{Soc}
 \opn\Im{Im} \opn\Ker{Ker} \opn\Coker{Coker} \opn\Am{Am}
 \opn\Hom{Hom} \opn\Tor{Tor} \opn\Ext{Ext} \opn\End{End}
 \opn\Aut{Aut} \opn\id{id}
 
 \opn\nat{nat}
 \opn\pff{pf}
 \opn\Pf{Pf} \opn\GL{GL} \opn\SL{SL} \opn\mod{mod} \opn\ord{ord}
 \opn\Gin{Gin} \opn\Hilb{Hilb}\opn\sort{sort}
 \opn\aff{aff} \opn
 \con{conv} \opn\relint{relint} \opn\st{st}
 \opn\lk{lk} \opn\cn{cn} \opn\core{core} \opn\vol{vol}  \opn\inp{inp} \opn\nilpot{nilpot}
 \opn\link{link} \opn\star{star}\opn\lex{lex}\opn\set{set} \opn\rev{rev} \opn\ini{in}
 \opn\gr{gr}

 \def\pot#1#2{#1[\kern-0.28ex[#2]\kern-0.28ex]}

 \opn\dirlim{\underrightarrow{\lim}}
 \opn\inivlim{\underleftarrow{\lim}}
 \def\Implies{\ifmmode\Longrightarrow \else
         \unskip${}\Longrightarrow{}$\ignorespaces\fi}
 \def\implies{\ifmmode\Rightarrow \else
         \unskip${}\Rightarrow{}$\ignorespaces\fi}
 \def\iff{\ifmmode\Longleftrightarrow \else
         \unskip${}\Longleftrightarrow{}$\ignorespaces\fi}

 \let\:=\colon
 \newtheorem{Theorem}{Theorem}[section]
 \newtheorem{Lemma}[Theorem]{Lemma}
 
 \newtheorem{Proposition}[Theorem]{Proposition}

 \newtheorem{Definition}[Theorem]{Definition}

 \newtheorem{Question}[Theorem]{Question}
 \let\epsilon\varepsilon
 \let\kappa=\varkappa
 \def\qed{\ifhmode\textqed\fi
       \ifmmode\ifinner\quad\qedsymbol\else\dispqed\fi\fi}
 \def\textqed{\unskip\nobreak\penalty50
        \hskip2em\hbox{}\nobreak\hfil\qedsymbol
        \parfillskip=0pt \finalhyphendemerits=0}
 \def\dispqed{\rlap{\qquad\qedsymbol}}
 
 \opn\dis{dis}
 \def\pnt{{\raise0.5mm\hbox{\large\bf.}}}
 
 \opn\Lex{Lex}

\begin{document}

 \title {Koszul binomial edge ideals of  pairs of graphs }

 \author {Herolistra Baskoroputro, Viviana Ene, Cristian Ion}

\address{Herolistra Baskoroputro, Abdus Salam School of Mathematical Sciences (GC University Lahore), 68-B New Muslim Town, Lahore 54600, Pakistan} \email{h.baskoroputro@gmail.com}

\address{Viviana Ene, Faculty of Mathematics and Computer Science, Ovidius University, Bd.\ Mamaia 124,
 900527 Constanta, Romania} \email{vivian@univ-ovidius.ro}

\address{Cristian Ion, University Dun\u area de Jos of Gala\c ti, Romania} \email{cristian.adrian.ion@gmail.com}

\thanks{Part of this paper was done while the first author visited Constanta. He would like to thank the Faculty of Mathematics and Computer Science of the Ovidius University of Constanta for  hospitality. The visit of the author to Constanta was supported by ASSMS and mainly funded by HEC Pakistan.}



\maketitle

\begin{abstract}
We study the Koszul property of a standard graded $K$-algebra $R$ defined by the binomial edge ideal  of a pair of graphs 
$(G_1,G_2)$. We show that the following statements are equivalent: (i) $R$ is Koszul; (ii) the defining ideal $J_{G_1,G_2}$ of $R$ has a quadratic Gr\"obner basis; (iii)  the graded maximal ideal of $R$ has linear quotients with respect to a suitable order of its generators.
\end{abstract}

\section{Introduction}

A standard graded  $K$--algebra $R$ is Koszul if its residual field $K$ has a linear resolution over $R.$ Koszul algebras occur frequently in combinatorial and geometric contexts. They were first introduced and studied  by Priddy \cite{P}. For a  nice survey on fundamental  results and open questions regarding Koszul algebras, we refer the reader to \cite{CDR}.

Let $R=S/I$ where $S=K[x_1,\ldots,x_n]$ is the polynomial ring over a field $K$ and $I\subset S$ is a homogeneous ideal. It is well-known 
that if $I$ has a quadratic Gr\"obner basis with respect to a coordinate system of $S_1$ and some monomial order on $S,$ that is, $R$ is $G$--quadratic, then $R$ is Koszul. 
On the other hand, if $R$ is Koszul, then $I$ is generated by quadrics. If $I$ is generated by quadratic monomials, then $R=S/I$ is Koszul. 
For the proofs of all these statements, one may consult, for example, \cite[Section 6.1]{EH}. Therefore, we have the following implications:
\begin{center}
$I$ has a quadratic Gr\"obner basis $\Longrightarrow S/I$  is Koszul $\Longrightarrow I$ is generated by quadrics.
\end{center}

There are examples which show that none of the above implications can be reversed; see \cite{CDR}, \cite[Section 6.1]{EH},  and the references therein.

\medskip

Nice  Koszul algebras arising from combinatorics are the ones defined by binomial edge ideals. 

Let $G$ be a simple graph on the vertex set $[n]$ with the edge set $E(G)$ and $S_G=K[x_1,\ldots,x_n,y_1,\ldots,y_n]$ the polynomial ring in $2n$ variables over the field $K.$ For $1\leq i<j\leq n,$ we set $f_{ij}=x_iy_j-x_jy_i.$ The binomial edge ideal of $G$ is 
$$J_G=(f_{ij}: \{i,j\}\in E(G))\subset S_G.$$ Since $J_G$ is generated by quadrics, it is natural to classify the graphs $G$ with the property that the algebra 
defined by $J_G$ is Koszul. A graph $G$ is called Koszul (over $K$) if the $K$--algebra $R_G=S_G/J_G$ is so. Koszul graphs have been studied in \cite{EHH1}. In that paper, it was shown that the following implications hold:
\begin{eqnarray}\nonumber\label{eq2}
G \text{ is closed,\ equivalently, } J_G \text{ has a quadratic Gr\"obner basis } \\ \nonumber
 \Longrightarrow G\text{ is Koszul}\Longrightarrow G\text{ is chordal and claw-free}.
\end{eqnarray}

We recall the combinatorial definition of closed graphs in Section~\ref{one}.
In \cite{EHH1}, one finds examples which show that none of the above implications can be reversed. Furthermore, in \cite[Section 3]{EHH1}, a classification of the Koszul graphs with  clique complex of dimension at most two is given in pure combinatorial terms.

\medskip

Koszul filtrations were introduced in \cite{CTV}. Let $R$ be a standard graded $K$--algebra with graded maximal ideal $\mm.$ A Koszul filtration of $R$ is a family $\F$ of ideals of $R$ generated by linear forms with the following properties:
\begin{itemize}
	\item [(i)] $\F$ contains the zero ideal and the maximal ideal $\mm;$
	\item [(ii)] for every non-zero ideal $I\in \F$ there exists $J\in \F$ such that $J\subset I$ and $I/J$ is a cyclic module whose annihilator, namely $J:I,$ belongs to $\F$.
\end{itemize}

When the set of all ideals generated by subsets of variables form a Koszul filtration of $R,$ then $R$ is called c-universally Koszul.

In \cite[Proposition 1.2]{CTV} it was shown that all the ideals of $\F$ have a linear resolution. In particular, it follows that 
$R$ is Koszul if it admits a Koszul filtration. However, there are examples of Koszul algebras which do not possess a Koszul filtration; see \cite{CTV}. 

In \cite{EHH2} it was proved that if the defining ideal $I$ of the standard graded $K$--algebra $R=S/I$ has a quadratic Gr\"obner basis with respect to the reverse lexicographic order induced by the natural order of the variables, then, for every $i,$ the ideal quotient 
$(I,x_n,x_{n-1},\ldots,x_{i+1}):x_i$ is generated, modulo $I$, by linear forms. One may easily find examples which show that even if $I$ 
has a quadratic Gr\"obner basis with respect to the reverse lexicographic order, then the ideals $(\bar{x}_n,\ldots,\bar{x}_{i+1}):\bar{x}_i$
are not generated by variables. (Here $^-$ denotes the residue class modulo $I.$) However, in case  of algebras defined by binomial edge ideals, by 
\cite[Theorem 2.1]{EHH2}, all the ideals $(\bar{x}_n,\ldots,\bar{x}_{i+1}):\bar{x}_i$ in $R_G=S_G/J_G$ have linear quotients if and only if $G$ is a closed graph. 

\medskip

The goal of this paper is to study Koszul algebras defined by binomial edge ideals of  pairs of graphs. 

Let $m,n\geq 3$ be some integers and $G_1,G_2$ simple graphs on the vertex sets $[m]$ and $[n],$  with edge sets $E(G_1), E(G_2)$, respectively. Let $S=K[X]$ be the polynomial ring in the variables $x_{ij}$ where $1\leq i\leq m$ and $1\leq j\leq n.$ For 
$1\leq i<j\leq m$ and $1\leq k<\ell\leq n$ such that $e=\{i,j\}\in E(G_1)$ and $f=\{k,\ell\}\in E(G_2)$, $p_{ef}$ denotes the $2$--minor 
of the matrix $X=(x_{ij})_{\stackrel{1\leq i\leq  m}{1\leq j\leq n}}$ determined by the rows $i,j$ and the columns $k,\ell$ of $X.$ Thus,
\[
p_{ef}=[ij| k\ell]=x_{ik}x_{j \ell}-x_{i\ell} x_{jk}.
\]

The binomial edge ideal of the pair $(G_1,G_2)$ is defined as 
\[
J_{G_1,G_2}=(p_{ef}: e\in E(G_1), f\in E(G_2)).
\]

This ideal was introduced in \cite{EHHQ}. In  \cite[Theorem 1.3]{EHHQ}, it was shown that $J_{G_1,G_2}$ has a quadratic Gr\"obner basis with respect to the lexicographic order induced by $x_{11}>x_{12}>\cdots > x_{1n}>x_{21}>\cdots >x_{2n}>\cdots >x_{m1}>\cdots >x_{mn}$ if and only if one of the graphs is closed and the other one is complete. 

A pair of graphs $(G_1,G_2)$ is called {\em Koszul} if the algebra $R=S/J_{G_1,G_2}$ is Koszul. The above mentioned statement implies that 
if $G_1$ is closed and $G_2$ is complete or vice-versa, the pair $(G_1,G_2)$ is Koszul. The surprising fact is that the converse is also true. Actually, we prove even more. Namely, in Theorem~\ref{Koszul} we show that the following statements are equivalent:

\begin{itemize}
	\item [(i)]  The pair of graphs $(G_1,G_2)$ is Koszul;
	\item [(ii)] $G_1$ is closed and $G_2$ is complete or vice-versa;
	\item [(iii)] The graded maximal ideal of $R$ has linear quotients with respect to a suitable order of its generators.
\end{itemize}

The paper is organized as follows. In Section~\ref{one} we provide a combinatorial characterization of closed graphs in Theorem~\ref{closed} which is needed to prove Theorem~\ref{Koszul}. The statement of Theorem~\ref{closed} appeared in \cite[Theorem 3.6]{Cru} in an equivalent form. We give here a different proof to that given in \cite{Cru} which might be of interest for algebraists since it does not involve many concepts and results of combinatorics. Our proof essentially uses only Dirac's theorem on chordal graphs \cite{D}.

Section~\ref{two} contains the main theorem of this paper, namely Theorem~\ref{Koszul}, which provides various characterizations of the Koszul pairs of graphs. Finally, we show that $R=S/J_{G_1,G_2}$ is c-universally Koszul if and only if $G_1$ and $G_2$ are complete graphs.

\section{A combinatorial characterization of closed graphs}
\label{one}

Closed graphs were   considered in \cite{HHHKR}. We recall here the definition. 

Let $G$ be a simple graph with $n$ vertices and with edge set $E(G)$. The graph $G$ is called {\em closed} if there exists a labeling 
of its vertices with labels from $1$ to $n$ such that the following condition is fulfilled: for any $i<j<k$ or $i>j>k,$ if 
$\{i,j\}, \{i,k\}$ are edges of $G,$ then $\{j,k\}$ is an edge as well. In fact, as it was shown in \cite[Theorem  1.1]{HHHKR}, a given labeling of $G$ satisfies the above condition if and only if $J_G$ has a quadratic Gr\"obner basis with respect to the lexicographic order on $S_G$ induced by the natural order of the variables, that is, $x_1>x_2>\cdots>x_n>y_1>\cdots >y_n.$

In \cite[Theorem 3.4]{CR}, it was proved that a graph $G$ is closed if and only if the associated binomial ideal $J_G\subset S_G=K[x_1,\ldots,x_n,y_1,\ldots,y_n]$ has a quadratic Gr\"obner basis with respect to some monomial order in the given coordinates of $(S_G)_1$. 

Later on, it was discovered that closed graphs are actually proper interval graphs (PI graphs in brief) which are known in combinatorics for a long time. For the original definition of the PI graphs and various characterizations of them we refer the reader to 
\cite{Cru, CR, Fi, Ga,  H,  LO, R}. In this paper, we will use the closed graph terminology.

It is easily seen  that any closed graph is chordal, that is, it has no induced cycle of length greater than $3$ and claw-free which means that  it has no induced graph isomorphic to the graph on the vertex set $[4]$ with edges $\{1,2\},\{1,3\},\{1,4\}$. In addition, if 
$G$ is closed, then any induced subgraph of $G$ must be closed. On the other hand, one may easily show that  the graphs of 
Figure~\ref{H1andH2} are not closed, but they are chordal and claw-free. These two graphs will play an important role in the combinatorial characterization of closed graphs which we are going to use in the next section.

\begin{figure}[hbt]
\begin{center}
\psset{unit=0.8cm}
\begin{pspicture}(1,1)(5,7)
\rput(-3,1){
\pspolygon(2,2)(3,3.71)(4,2)
\psline(3,3.71)(3,5.2)
\psline(0.6,1.1)(2,2)
\psline(4,2)(5.6,1.1)
\rput(2,2){$\bullet$}
\rput(3,3.71){$\bullet$}
\rput(4,2){$\bullet$}
\rput(3,5.2){$\bullet$}
\rput(0.6,1.1){$\bullet$}
\rput(5.6,1.1){$\bullet$}
\rput(3.3,0.4){$H_1$}
}
\rput(4,1.5){
\psset{unit=1.5cm}
\pspolygon(0.5,0.5)(1.5,0.5)(1,1.4)
\pspolygon(1.5,0.5)(2,1.4)(2.5,0.5)
\pspolygon(1,1.4)(1.5,2.3)(2,1.4)

\rput(0.5,0.5){$\bullet$}
\rput(1.5,0.5){$\bullet$}
\rput(1,1.4){$\bullet$}
\rput(2,1.4){$\bullet$}
\rput(2.5,0.5){$\bullet$}
\rput(1.5,2.3){$\bullet$}
\rput(1.5,-0.06){ $H_2$}
}
\end{pspicture}
\end{center}
\caption{Non-closed graphs}
\label{H1andH2}
\end{figure}
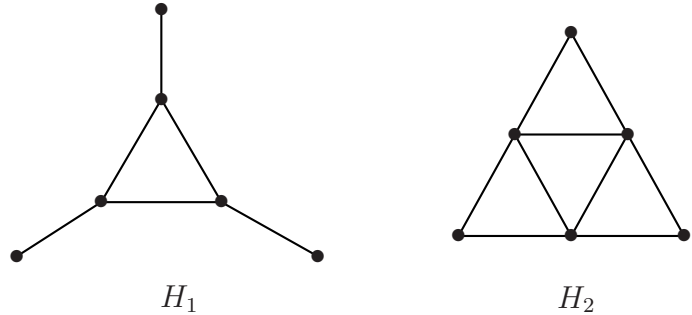

By a theorem of Dirac \cite{D}, any chordal graph $G$ has a perfect elimination order which means that its vertices can be labeled with 
the numbers $1,\ldots,n$ such that for every $j,$ the set $C_j=\{i: i<j \text{ and }\{i,j\}\in E(G)\}$ is a clique of $G.$ We recall that a clique of a graph $G$ means a complete subgraph of $G.$ The set of cliques of a graph $G$ forms a simplicial complex $\Delta(G)$ which is called the clique complex of $G.$ Its facets are the maximal cliques of $G.$ By using Dirac's theorem and an inductive argument, in \cite{EHH3},  the following characterization of closed graphs was given.

\begin{Theorem}\cite{EHH3}
\label{characterization1} Let $G$ be a graph on $[n]$. The following conditions are equivalent:
\begin{enumerate}
\item[{\em (a)}] $G$ is closed;
\item[{\em (b)}] there exists a labeling of $G$ such that all facets of $\Delta(G)$ are intervals $[a,b]\subset [n]$.
\end{enumerate}
Moreover, if the equivalent conditions hold and the facets $F_1,\ldots,F_r$ of $\Delta(G)$ are labeled such that $\min(F_1)<\min(F_2)<\cdots <\min(F_r)$, then $F_1,\ldots,F_r$ is a leaf order of $\Delta(G)$.
\end{Theorem}

For more information about clique complexes, leaf order, and algebraic aspects of Dirac's theorem, one may consult \cite[Section 9.2]{HHBook}.

We are now ready to give a new characterization of the closed graphs. As we have already mentioned in Introduction, an equivalent statement appears in \cite{Cru}, but with a completely different proof.

\begin{Theorem}\label{closed}
Let $G$ be a connected  graph. Then $G$ is closed if and only if it is chordal, claw--free, and  none of the graphs depicted in Figure \ref{H1andH2} is 
an induced subgraph of $G.$
\end{Theorem}

\begin{proof}
If $G$ is closed, then any induced subgraph of $G$ is closed as well, hence $G$ must be chordal, claw-free, and  none of the graphs $H_1$ and $H_2$ can be an induced subgraph 
of $G$ since they are not closed.

We prove the converse by induction on the number of vertices of $G.$ If $G$ has two vertices, the statement is trivial. We assume now 
that $G$ is a connected chordal claw-free graph on the vertex set $[n]$, with $n\geq 3,$ and that the converse is true for graphs with $n-1
$ vertices. Since 
$G$ is chordal,  we may choose a 
perfect elimination order on $G.$ Then the vertex labeled with $n$ is obviously a free vertex, that is, the vertex $n$ belongs to exactly 
one maximal clique of $G.$

Let $G^\prime$ be the restriction of $G$ to the vertex set $[n-1]$. Then $G^\prime$ is clearly chordal claw--free and has no induced 
subgraph isomorphic to $H_1$ or $H_2.$ We claim that $G^\prime$ is also connected. Indeed, if $G^\prime$ has at least two connected 
components, as $G$ is connected, it follows that the vertex $n$ must belong to at least two maximal cliques of $G$, contradiction. 
Therefore, we may apply the inductive hypothesis to $G^\prime$ and conclude that $G^\prime$ is closed. By Theorem~\ref{characterization1}, 
it follows that we may relabel the vertices of $G^\prime$ with labels from $1$ to $n-1$ such that the facets of $\Delta(G^\prime)$ are $F_1
=[a_1,b_1],\ldots,F_r=[a_r,b_r]$ with $1=a_1<a_2<\ldots<a_{r}<b_r=n-1.$ 

We have to show that $G$ is closed, that is, that $G$ satisfies the condition (b) of Theorem~\ref{characterization1}.The idea of the proof 
is very simple, but the details need some work. We have to figure out where the vertex labeled with $n$ may be "located" such that we do 
not violate the hypothesis on $G$   and next, we show that, for each such location of the vertex  $n,$ one may relabel all the vertices of 
$G$ such that condition (b) of Theorem\ref{characterization1} holds, hence $G$ is closed.

Let us first assume that $G^\prime$ itself is a clique. If the vertex $n$ of $G$ is adjacent to all the vertices of $G^\prime$, then $G$ 
is a clique as well, thus it is closed. If not, then we may relabel the vertices of $G^\prime$ such that those which are adjacent to the 
vertex $n$ of $G$ have the largest labels among $1,\ldots,n-1$. Then, we get $\Delta(G)=\langle [1,n-1],[a,n] \rangle$ for some $1<a\leq 
n-1,$ thus $G$ is a closed graph with two maximal cliques.

We consider now the case when $G^\prime$ has two maximal cliques, say, $\Delta(G^\prime)=\langle F_1=[1,b],F_2=[a,n-1] \rangle$ for some $1
<a\leq b< n-1.$ Let $i_1,\ldots, i_{\ell}\in [n-1]$ be the vertices of $G^\prime$ which belong to the maximal clique of $G$ which 
contains the vertex $n.$ In other words, $i_1,\ldots, i_{\ell}$ are all the  vertices adjacent to $n$ in $G.$ As $i_1,\ldots, i_{\ell}$ 
form also a clique, all these vertices must be contained in one of the two cliques of $G
^\prime.$ We may assume that $i_1,\ldots, i_{\ell}\in F_2.$ Otherwise, we reduce to this case by relabeling the vertices of $G^\prime$ as 
follows: $i\mapsto n-i$ for $1\leq i\leq n-1.$ If $\{i_1,\ldots, i_{\ell}\}=F_2,$ then $\Delta(G)=\langle [1,b],[a,n]\rangle$, hence $G$ 
is closed. 

Now we assume that $\{i_1,\ldots, i_{\ell}\}\subsetneq F_2.$  If all the vertices $i_1,\ldots, i_{\ell}$ are free in $F_2,$ then we may 
relabel all the  free vertices of $F_2$ such that $\{i_1,\ldots, i_{\ell}\}=\{n-1,n-2,\ldots,n-\ell\}.$ It follows that 
$\Delta(G)=\langle [1,a], [b,n-1],[n-\ell,n] \rangle$, thus $G$ is closed. We have to treat now the case when at least one of the vertices 
$i_1,\ldots, i_{\ell}$, let us say $i_1,$  belong to $F_1\cap F_2.$ If there is a free vertex  $k\in F_2$ which is not adjacent to $n,$ 
then we get an induced claw graph in $G$ with the edges $\{1,i_1\}, \{i_1,n\},\{i_1,k\}$ which is impossible. Therefore, all the free 
vertices of $F_2$ are contained in the set $\{i_1,\ldots, i_{\ell}\}.$ In this case we may relabel the vertices in the intersection 
$F_1\cap F_2$ such that the set $\{i_1,\ldots, i_{\ell}\}$ is an interval of the form $[c,n-1]$ where $a<c\leq b.$ Consequently, 
$\Delta(G)=\langle [1,b],[a,n-1],[c,n] \rangle$, thus $G$ is closed. Here we have to mention that any permutation of the labels of the intersection vertices does not modify the intervals in $G^\prime$. 

Finally, we discuss the case when $\Delta(G^\prime)$ has at least three facets, that is, the facets of $\Delta(G^\prime)$ are 
$F_1=[a_1,b_1],\ldots,F_r=[a_r,b_r]$ with $1=a_1<a_2<\cdots<a_{r}<b_r=n-1$ and $r\geq 3.$ Let, as before, $i_1,\ldots, i_{\ell}$ be the 
vertices adjacent to the vertex $n.$ Since $i_1,\ldots, i_{\ell}$ form a clique in $G^\prime,$ there must be a maximal clique of $G^\prime$
 which contains the set $\{i_1,\ldots, i_{\ell}\}.$
We distinguish two cases.

{\em Case 1.} $\{i_1,\ldots, i_{\ell}\}\subseteq F_r.$ 

If we have equality, then clearly $G$ is closed since 
$\Delta(G)=\langle F_1,\ldots,F_{r-1}, F_r\cup\{n\} \rangle$. 

Let now  $\{i_1,\ldots, i_{\ell}\}\subsetneq F_r.$ We proceed further as we 
have already done in the case when $G^\prime$ had two cliques. Indeed, if $i_1,\ldots, i_{\ell}$ are all free vertices of $F_r,$ then we 
may relabel the free vertices of $F_r$ such that $\{i_1,\ldots, i_{\ell}\}=\{n-\ell,n-\ell+1,\ldots,n-1\}.$ With respect to this new 
labeling, we get $\Delta(G=\langle F_1,\ldots,F_r, [n-\ell, n]\rangle),$ thus $G$ is closed. The difference to the preceding case when 
$\Delta(G^\prime)$ had two cliques consists in the fact that $F_r$ may have non-empty intersection with  several maximal cliques of 
$G^\prime$. We may choose the smallest integer 
$j$
such that there exists an index in the set $\{i_1,\ldots,i_\ell\}$, say $i_1,$ such that $i_1\in F_j\cap F_r.$ We claim that in this case, the set $F_r\setminus F_j$ must 
be contained in $\{i_1,\ldots, i_{\ell}\}.$ Indeed, let us assume that there exists $k\in F_r\setminus F_j$ such that $k$ is not adjacent 
to the vertex $n$ of $G.$ If follows that  $G$  has the induced claw with the edges $\{\min F_j, i_1\},\{i_1,n\},\{i_1,k\}$, which is 
impossible. Thus $F_r\setminus F_j\subset \{i_1,\ldots, i_{\ell}\}$. Then we may relabel (if necessary) the vertices of $F_j\cap F_r$ such 
that the set $\{i_1,\ldots, i_{\ell}\}$ is an interval of the form $[c,n-1]$ where $a_r < c\leq b_j.$ With respect to this new labeling, 
the maximal cliques of $G$ are the intervals $F_1,\ldots,F_{r-1}, F_r,$ and $F_{r+1}=[c,n]$, hence $G$ is  closed.

If $\{i_1,\ldots, i_{\ell}\}\subseteq F_1,$ then we may reduce to the case that we have just discussed by reversing the labels of $G^\prime$, namely:
$i\mapsto n-i$ for $1\leq i\leq n-1.$ Therefore it remains to discuss the following case.

{\em Case 2.} $\{i_1,\ldots, i_{\ell}\}\subseteq F_i$ for some $2\leq i\leq r-1.$ 

If we have equality, namely 
$\{i_1,\ldots, i_{\ell}\}= F_i$ and $F_{i-1}\cap F_{i+1}=\emptyset$, in other words, $F_i$ has free vertices,  then we may relabel the 
vertices of $F_i\cup\{n\}$ and of $F_{i+1},\ldots, F_r$ such that 
$\Delta(G)=\langle F_1=[a_1,b_1],\ldots,F_{i-1}=[a_{i-1},b_{i-1}], F^\prime_i=[a_i,b_i+1],F^\prime_{i+1}=[a_{i+1}+1, b_{i+1}+1],\ldots,F^
\prime_r=[a_r+1,b_r+1=n]\rangle$. The case when  $F_i$ has no free vertex, that is $F_{i-1}\cap F_{i+1}\neq\emptyset,$ and 
$\{i_1,\ldots, i_{\ell}\}= F_i$ cannot occur. Indeed, let $j\in F_{i-1}\cap F_{i+1}$ and set 
$p=\min\{t: j\in F_t\}$, $q=\max\{t:  j\in F_t\}.$ Then $G$ has the induced claw with edges $\{\min F_p,j\},\{j,n\},\{j,\max F_q\}$ as induced 
graph, which is impossible.

Let now $\{i_1,\ldots, i_{\ell}\}\subsetneq F_i.$ We split the rest of the proof into two subcases.

{\em Subcase 2 (a).} $F_{i-1}\cap F_{i+1}=
\emptyset.$ Let us first assume that there exists vertices $p\in F_{i-1}\cap F_i$ and $q\in F_i\cap F_{i+1}$ which are not adjacent to $n.$
Now we look at the possible neighbors of $n$. If there exists $j\in F_i\setminus(F_{i-1}\cup F_{i+1})$ such that $\{j,n\}\in E(G),$ then we get an induced subgraph of $G$ isomorphic to $H_1$ by choosing the triangle $\{j,p,q\}$ together with the 
edges $\{j,n\},\{\min F_{i-1},p\},\{q,\max F_{i+1}.\}$.  If there exists $j\in F_{i-1}\cap F_i$ such that $\{j,n\}\in E(G),$ then $G$ 
has the induced claw with edges $\{\min F_{i-1},j\}, \{j,n\},\{j,q\}.$ The case $j\in F_i\cap F_{i+1}$ is symmetric. 

Therefore, we have shown that if $F_{i-1}\cap F_{i+1}=\emptyset,$ then we must have $F_{i-1}\cap F_i\subset \{i_1,\ldots, i_{\ell}\}$ 
or $F_{i+1}\cap F_i\subset \{i_1,\ldots, i_{\ell}\}.$ Clearly, by symmetry, we may assume that $F_{i+1}\cap F_i\subset \{i_1,\ldots, i_{\ell}\}.$ If there exists $p\in F_i\cap F_{i-1}$ which is not adjacent to $n$, we get the induced claw of $G$ with edges $$\{p,\min(F_i\cap F_{i+1})\},\{n,\min(F_i\cap F_{i+1})\},\{\min(F_i\cap F_
{i+1}), \max F_{i+1}\}.$$ Thus, we have shown that 
$\{i_1,\ldots, i_{\ell}\}$ must contain $(F_i\cap F_{i-1})\cup (F_i\cap F_{i+1}).$ As $\{i_1,\ldots, i_{\ell}\}\subsetneq F_i.,$ there must exist a free vertex of $F_i$ which is not adjacent to $n.$ Then, we get the induced claw in $G$ with the edges 
$\{u,\min(F_i\cap F_{i+1})\},\{n,\min(F_i\cap F_{i+1})\},\{\min(F_i\cap F_{i+1}),\max F_{i+1}\}$. 

Summarizing, we have shown that if 
$F_{i-1}\cap F_{i+1}=\emptyset$, then $G$ must contain an induced graph isomorphic either to a claw or to $H_1$, which impossible.

{\em Subcase 2 (b).} $F_{i-1}\cap F_{i+1}\neq \emptyset.$ We will show that also this subcase cannot  occur.
If there exists a vertex 
$j\in F_{i-1}\cap F_{i+1}$ which is adjacent to $n,$ then $G$ has an induced claw with the edges $\{\min F_{i-1},j\},\{j,n\},\{j,\max F_{i+
1}\}$, contradiction. Consequently, $n$ cannot be adjacent to any vertex of $F_{i-1}\cap F_{i+1}.$ 

Let now $j\in (F_i \cap F_{i-1})\setminus F_{i+1}$ adjacent to $n.$ If there is no vertex adjacent to $n$ among the vertices of $F_i\cap F_{i+1}$, then we 
get the induced claw of $G$ with the edges $$\{\min F_{i-1},j\},\{j,n\},\{j,\max F_i\}.$$ This implies that all the vertices 
in the set $(F_i\cap F_{i+1})\setminus F_{i-1}$ must be adjacent to $n.$ but, in this case, we reach a contradiction in the following way. 
Let $t\in (F_i\cap F_{i+1})\setminus F_{i-1}.$ The induced subgraph of $G$ with the triangles 
$$\{\min F_{i-1},j,\max F_{i-1}\},\{j,\max F_{i-1},t\},\{\max F_{i-1},t,\max F_{i+1}\},\text{ and  }\{n,j,t\}$$ is isomorphic to $H_2,$ contradiction to the 
hypothesis on $G. $ 

We end this subcase and the whole proof by observing that the situation when we choose $j\in (F_i\cap F_{i+1})\setminus F_{i-1}$ adjacent to $n$ is symmetric 
to the above one.
\end{proof}

\section{Koszul pairs of graphs}
\label{two}

In this section we state and prove the main theorem of this paper.

\medskip

Let $m,n\geq 3$ be integers and $G_1,G_2$ graphs on the vertex sets $[m],[n],$ respectively. Let $X$ be the $(m\times n)$--matrix with entries $\{x_{pq}\}_{\stackrel{1\leq p\leq m}{ 1\leq q\leq n}}$ and $S=K[X]$ the polynomial ring over $K$ with indeterminates  
$\{x_{pq}\}_{\stackrel{1\leq p\leq m}{ 1\leq q\leq n}}$.
 Let $J_{G_1,G_2}$ be the binomial edge 
ideal of the pair $(G_1,G_2)$. As we have already mentioned, this ideal is generated by all the minors $p_{ef}=[ij|k\ell]$ of the generic 
matrix $X$ with $\{i,j\}\in E(G_1)$ and $\{k,\ell\}\in E(G_2).$ 
The pair $(G_1,G_2)$ is called {\em Koszul} if the algebra $R=S/J_{G_1,G_2}$ is Koszul.

To begin with, we notice that we may reduce the study of  Kozulness of the algebra $R$ to the case when both graphs are connected. 

Indeed, let $G_{11},\ldots,G_{1p}$ be the connected components of $G_1$ and $G_{21},\ldots,G_{2q}$  the connected components of $G_2$. Then, \[\frac{S}{J_{G_1,G_2}} \cong \bigotimes_{r,s}\frac{S_{rs}}{J_{G_{1r},G_{2s}}}, \] where $S_{rs}=K[\{x_{tu}:t \in V(G_{1r}), u \in V(G_{2s})\}]$ for all $1 \leq r \leq p, 1 \leq s \leq q.$

We know that $S/J_{G_1,G_2}$ is Koszul if and only if each factor $S_{rs}/J_{G_{1r},G_{2s}}$ is Koszul by \cite[Proposition 2.1]{EHH2}.

\begin{Theorem}\label{Koszul}
Let $G_1,G_2$ be two connected graphs on the vertex sets $[m], [n]$ respectively, where $m,n\geq 3.$ The following statements are equivalent:
\begin{itemize}
\item [(i)]  The pair of graphs $(G_1,G_2)$ is Koszul;
	\item [(ii)] $G_1$ is closed and $G_2$ is complete or vice-versa;
	\item [(iii)] $J_{G_1,G_2}$ has a quadratic Gr\"obner basis with respect to the lexicographic order, $\lex,$ induced by 
	$x_{11}>x_{12}>\cdots > x_{1n}>x_{21}>\cdots >x_{2n}>\cdots >x_{m1}>\cdots >x_{mn}$.
	\item [(iv)] $J_{G_1,G_2}$ has a quadratic Gr\"obner basis with respect to the reverse lexicographic order, $\rev,$ induced by
	$x_{1n}>x_{2n}>\cdots > x_{mn}>x_{1 n-1}>\cdots >x_{m n-1}>\cdots >x_{11}>\cdots >x_{m1}$.
	\item [(v)] The graded maximal ideal of $R$ has linear quotients with respect to the following order of its generators:
	\[
	\overline{x}_{m1},\overline{x}_{m-1,1},\ldots, \overline{x}_{11},\overline{x}_{m2},\ldots, \overline{x}_{12}, \ldots, \overline{x}_{mn},
	\ldots, \overline{x}_{1n}.
	\] 
\end{itemize}
\end{Theorem}

\begin{proof} We first observe that (iii) and (iv) are obviously equivalent since, for $1\leq i<j\leq m$ and $1\leq k<\ell < n,$ we have
$\ini_{\lex}(x_{ik}x_{j \ell}-x_{i\ell} x_{jk}) = \ini_{\rev}(x_{ik}x_{j \ell}-x_{i\ell} x_{jk})$. This is because we do not change only the monomial order, but also the order of the variables.

Furthermore, (ii) and (iii) are equivalent by \cite[Theorem 1.3]{EHHQ}. Also, (iii)$\Rightarrow$(i) is a  known  general statement. 

In what follows we will prove (i)$\Rightarrow$(ii) and (ii)$\Leftrightarrow$(v). This will complete the whole proof of the theorem.

\subsection*{Proof of (i)$\Rightarrow$(ii)}

We begin by proving that at least one of the two graphs must be complete.

We first make the following remark. 
Let $G_1^\prime$ and $G_2^\prime$ be induced subgraphs of $G_1,G_2,$ respectively, and 
let $Y$ be the set of variables $x_{ij}$ with $i\in V(G_1^\prime)$ and $j\in V(G_2^\prime).$ Then, by the proof of \cite[Proposition 8]{Sara2}, it follows that 
$T=K[Y]/J_{G_1^\prime,G_2^\prime}$ is an algebra retract of $R=S/J_{G_1,G_2}$. Thus, by \cite[Proposition 1.4]{OHH}, $T$ must be a Koszul algebra, hence the pair $(G_1^\prime,G_2^\prime)$ is also Koszul. The same idea was used in \cite{EHH1} to show that an induced subgraph of a Koszul graph is also Koszul.

Let us assume that neither $G_1$ nor $G_2$ is complete. Then there exists two induced path subgraphs, $L_1\subset G_1$ and $L_2\subset G_2,$ each of them consisting  of two edges, say $E(L_1)=\{\{i,j\},\{j,k\}\}$ and $E(L_2)=\{\{p,q\},\{q,r\}\}.$ Let $Y$ be the subset of $X$ containing the variables $x_{ab}$ with  $a\in \{i,j,k\}$ and $ b\in \{p,q,r\}.$ Then, by the above remark, $T=K[Y]/J_{L_1,L_2}$ must be a Koszul algebra. But this is not true since $\beta_{35}^T(K)\neq 0$ as one may check by using Singular \cite{DGPS}. Indeed, the beginning of the resolution of $K$ over $T$ is the following:
\begin{verbatim}
           0     1     2     3     4     5
------------------------------------------
    0:     1     9    40   120   280   552
    1:     -     -     -     -     -     -
    2:     -     -     -     2    24   148
------------------------------------------
\end{verbatim}

Let us now assume that $G_2$ is complete. We  begin by proving that $G_1$ is Koszul. This will follow easily by applying again the above remark for the algebra retract $T^\prime=K[Z]/J_{G_1,f}$, where $f=\{1,2\}\in E(G_2)$ and $Z=\{x_{ij}:1\leq i\leq m, j\in \{1,2\}\}.$ Hence $T^\prime$ is Koszul. But $J_{G_1,f}$ is exactly the classical binomial edge ideal associated with $G_1$, thus $G_1$ is Koszul.

It remains to prove that $G_1$ is even closed. By \cite[Theorem 2.1]{EHH1}, it follows that $G_1$ is chordal and claw-free. 
In order to apply Theorem~\ref{closed} and derive the desired conclusion, it only remains to show that $G_1$ has no induced subgraph isomorphic to $H_1$ or $H_2.$

It is known that 
the graph $H_2$ of Figure~\ref{H1andH2} is not Koszul; see \cite[Page 133]{EHH1}. This implies that $G_1$ has no induced subgraph isomorphic to $H_2.$  Let us suppose that there exists  an induced subgraph of $G_1,$ say $G_1^\prime$ isomorphic to $H_1.$ Let $G_2^\prime$
be the complete subgraph of $G_2$ on the vertex set $[3].$ We denote by $U$ the set of all variables $x_{ij}$ with $i\in V(G^\prime_1)$ and 
$1\leq j\leq 3.$ Then, $T^{\prime\prime}=K[U]/J_{G_1^\prime,G_2^\prime}$ is an algebra retract of $R$ by the above remark. Hence 
$T^{\prime\prime}$ should be Koszul. But this is not true by the following lemma, hence $G_1$ has no induced subgraph isomorphic to $H_1$. This completes the proof of (i)$\Rightarrow$(ii).

\begin{Lemma}\label{net}
The pair of graphs $(H_1,K_3)$ is not Koszul.
\end{Lemma}

\begin{proof}
We label the vertices of $H_1$ as follows. We assign the labels $2,3,4$ to the vertices of the triangle. The additional edges are 
$\{1,2\}, \{3,5\},\{4,6\}.$ Let $H_1^\prime$ be the induced subgraph of $H_1$ on the vertex set $[5].$ $H_1^\prime$ is obviously a closed 
graph since it admits another labeling which is closed. Thus, the pair $(H_1^\prime,K_3)$ is Koszul since $J_{H_1^\prime, K_3}$ has a quadratic Gr\"obner basis by \cite[Theorem 1.3]{EHHQ}. 

Let 
$T=K[\{x_{ij}\}_{1\leq i\leq 6,\ 1\leq j\leq 3}]/J_{H_1,K_3}$ be the coordinate ring of the pair $(H_1,K_3)$ and 
$T^\prime=K[\{x_{ij}\}_{1\leq i\leq 5,\ 1\leq j\leq 3}]/J_{H^\prime_1,K_3}$ the coordinate ring of the pair $(H^\prime_1,K_3).$
Note that $T^\prime \cong T/(\bar{x}_{61},\bar{x}_{62},\bar{x}_{63})$ where $^-$ denotes the class modulo $J_{H_1,K_3}$. As $H^\prime_1$ 
is an induced graph in $H_1,$ it follows that $T^\prime$ is an algebra retract of $T.$ Therefore, by \cite[Proposition 1.4]{OHH}, it 
follows that $T$ is Koszul if and only if $T^\prime$ has a linear resolution over $T.$ But this is false since, as one may check with 
Singular, we have $\beta_{56}^T(T^\prime)=1\neq 0.$
\end{proof}

\subsection*{Proof of (ii)$\Leftrightarrow$(v)}

We begin with (ii)$\Rightarrow$(v). In the hypothesis (ii), by Theorem~\ref{Koszul}, it follows that $J_{G_1,G_2}$ has a quadratic Gr\"obner basis with respect to the reverse lexicographic order induced by
	$x_{1n}>x_{2n}>\cdots > x_{mn}>x_{1, n-1}>\cdots >x_{m, n-1}>\cdots >x_{11}>\cdots >x_{m1}$.
Then, by applying \cite[Theorem 1.3]{EHH3}, it follows that the sequence \[x_{m1},x_{m-1,1},\ldots,x_{11},x_{m2},\ldots,x_{12},\ldots,x_{mn},\ldots,x_{1n}\] has  linear quotients modulo $J_{G_1,G_2}$. 

(v)$\Rightarrow$(ii). We assume that \[ {x_{m1}},{x_{m-1,1}},\ldots,{x_{11}},{x_{m2}},\ldots,{x_{12}},\ldots,
{x_{mn}},\ldots,{x_{1n}} \] has linear quotients modulo $J_{G_1,G_2}$.

In the first step, we show that at least one of the two graphs must be complete. Let us assume that none of them is complete. Then there exist $L_i$ an induced path with $3$ vertices in $G_i$ for $i=1,2$.

Let $i < j < k$ be the vertices of $L_1$ and $p<q<r$ the vertices of $L_2$. We claim that $\overline{x}^2_{iq}\overline{x}_{kr}$ is a 
minimal generator of the ideal quotient \[Q=(\overline{x}_{m1},\dots,\overline{x}_{11},\ldots,\overline{x}_{mp},\ldots,\overline{x}_{kp},
\ldots,\overline{x}_{j+1,p}):\overline{x}_{jp}.\]
If we prove the above claim, we reach a contradiction to our hypothesis. It follows that at least one of the two graphs must be complete.

We have
\begin{eqnarray}\label{quadratic generator}
& \overline{x}_{iq}^2\overline{x}_{kr}\overline{x}_{jp} = \overline{x}_{iq}\overline{x}_{ip}\overline{x}_{jq}\overline{x}_{kr} =\\  \nonumber
& = \overline{x}_{iq}\overline{x}_{ip}\overline{x}_{jr}\overline{x}_{kq} = \\  \nonumber
& =\overline{x}_{ip}\overline{x}_{kq}\overline{x}_{ir}\overline{x}_{jq}= \\ \nonumber
& = \overline{x}_{kq}\overline{x}_{ir}\overline{x}_{iq}\overline{x}_{jp} = \\ \nonumber
& =\overline{x}_{ir}\overline{x}_{iq}\overline{x}_{jq}\overline{x}_{kp}. \nonumber
\end{eqnarray}


 Equality (\ref{quadratic generator}) yields $\overline{x}^2_{iq}\overline{x}_{kr} \in Q$. We have to show that $\overline{x}^2_{iq}\overline{x}_{kr}$ is a minimal generator of $Q$.
By assumption, $Q$ is generated by linear forms, hence if $\overline{x_{iq}^2}\overline{x_{kr}}$ is not a minimal generator, then there must  exist the forms $\overline{\ell},\overline{v}$ such that \begin{equation}\label{form} \overline{x}^2_{iq}\overline{x}_{kr} =\overline{l}\overline{v}, \end{equation} where $\overline{l}$ is a minimal generator of $\mathrm{Q}$ and $v \in S$ is a homogeneous polynomial of degree 2.

We consider the $\mathbb{Z}^{m+n}$-multigrading on $S$ by defining \[m-\deg(x_{ij})= \varepsilon_{i,j+m}, 1\leq i\leq m, 1\leq j\leq n,\] where $\varepsilon_{i,j+m}=\varepsilon_i + \varepsilon_{j+m}$, and $\varepsilon_k$ is the $k$-th element in the canonical basis of $\mathbb{Z}^{m+n}$.

The ideal $J_{G_1,G_2}$ and, consequently, the algebra $R = S/J_{G_1,G_2}$ are homogeneous with respect to this grading. Then the forms $\ell,v$ are also $\mathbb{Z}^{m+n}$-homogenous, thus, by equality (\ref{form}), it follows that \[m-\deg(\ell)\leq 2 \varepsilon_{i,q+m}+\varepsilon_{k,m+r}\] componentwise. The same holds for $v$.

Equation (\ref{form}) implies that \[x^2_{iq}x_{kr} - \ell v \in J_{G_1,G_2},\] thus,\begin{equation}\label{grade}x^2_{iq}x_{kr} -\ell v =
\sum_{e \in E(G_1), f \in E(G_2)}h_{ef} p_{ef},\end{equation} for some polynomial $h_{ef} \in S.$

In equality (\ref{grade}), we substitute $x_{uv}$ by $0$ for every pair $(u,v) \not\in \{i,j,k\} \times \{p,q,r\}.$ Te conditions on the multidegrees of $\ell$ and $v$ imply that $x_{iq}^2x_{kr} - v\ell \in J_{L_1,L_2}.$
On the other hand, $\overline{\ell}$ is a minimal generator of $Q$, hence \[\ell x_{jp}\in 
(J_{G_1,G_2},x_{m1},\ldots,x_{11},\ldots,x_{mp},\ldots,x_{kp},\ldots,x_{j+1,p}).\]

By using  again  the condition on the  multidegree of $\ell$, it follows that \[\ell x_{jp} \in (J_{L_1,L_2},x_{kp} )\] in the subring $S' = K[X']$, where $X' =
\left(\begin{array}{ccc}
x_{ip} & x_{iq} & x_{ir} \\
x_{jp} & x_{jq} & x_{jr} \\
x_{kp} & x_{kq} & x_{kr} 
\end{array}\right).
$
This implies that $\ell$ is a minimal generator of $(J_{L_1,L_2},x_{kp}):x_{jp}$ in $S'$. By using Singular, we may easily see that there is no minimal generator of $(J_{L_1,L_2},x_{kp}):x_{jp}$ which satisfies the multidegree inequality of $\ell$.

Therefore, we have proved that $Q$ does not have linear quotients. Consequently, at least one of the two graphs must be complete.

We first choose $G_2$ to be complete. We have to show that $G_1$ is closed with respect to its given labeling. By hypothesis, we know that 
for every $2\leq i\leq m,$ the ideal quotient 
\[
(\overline{x}_{m1},\ldots,\overline{x}_{11},\ldots, \overline{x}_{m, n-2},\ldots,\overline{x}_{1, n-2},\overline{x}_{m, n-1}\ldots, \overline{
x}_{i, n-1}): \overline{x}_{i-1, n-1}\]  is generated by linear forms in $R.$ This is equivalent to saying that 
\[
(x_{m1},\ldots,x_{11},\ldots, x_{m,n-2},\ldots,x_{1,n-2},x_{m,n-1},\ldots,x_{i,n-1},J_{G_1,G_2}): x_{i-1,n-1}
\]
is generated by linear forms modulo $J_{G_1,G_2}.$

We have 
\[
(x_{m1},\ldots,x_{11},\ldots, x_{m,n-2},\ldots,x_{1,n-2},x_{m,n-1},\ldots,x_{i,n-1},J_{G_1,G_2}): x_{i-1,n-1}=
\]
\[
(x_{m1},\ldots,x_{11},\ldots, x_{m,n-2},\ldots,x_{1,n-2})+(x_{m,n-1},\ldots,x_{i,n-1},J_{G_1,\{n-1,n\}}): x_{i-1,n-1}
\]
and the last term in the above sum is generated by linear forms modulo $J_{G_1,\{n-1,n\}}$ if and only if $G_1$ is closed with respect to its given labeling  by \cite[Theorem 1.6]{EHH2}.

It remains to consider the case when $G_1$ is complete. We have to show that $G_2$ is closed with respect to its given labeling. Assume that this is not the case and that there exist $i<j<k$ or $i>j>k$ such that $\{i,j\}, \{i,k\}\in E(G_2)$ and $\{j,k\}\notin E(G_2).$

It is enough to make the proof for  $i<j<k$ such that $\{i,j\}, \{i,k\}\in E(G_2)$ and $\{j,k\}\notin E(G_2)$ since the case 
$i>j>k$ is symmetric. We only need to exchange the roles of $i$ and $k.$ Let us consider the minor 
\[g=[m-2\ \ m-1|j k]=x_{m-2,j}x_{m-1,k}-x_{m-2,k}x_{m-1,j}.\] We observe that
\[
\overline{x}_{mi}\overline{g}=\overline{x}_{mi}\overline{x}_{m-2,j}\overline{x}_{m-1,k}-
\overline{x}_{mi}\overline{x}_{m-2,k}\overline{x}_{m-1,j}=
\]
\[
\overline{x}_{m-2,i}\overline{x}_{mj}\overline{x}_{m-1,k}-
\overline{x}_{m-1,i}\overline{x}_{mj}\overline{x}_{m-2,k}=0 (\mod (J_{G_1,G_2})).
\] 
In the above relations, we used that the minors $[m-2\ \ m|i j], [m-2\ \ m-1|i k]$ belong to $J_{G_1,G_2}$ since 
$\{i,j\}, \{i,k\}\in E(G_2).$ The above calculation shows that $\overline{g}$ belongs to ideal quotient
\[
Q=(\overline{x}_{m1},\ldots, \overline{x}_{11},\ldots,\overline{x}_{m,i+1},\ldots, \overline{x}_{1,i+1}): 
\overline{x}_{mi}. 
\] We claim that $\overline{g}$ is a minimal generator of $Q.$ This will then give a contradiction to our hypothesis and completes the 
proof for $i<j<k.$ The proof of the claim uses arguments similar to those of the previous part of this proof. Let us assume that 
$\overline{g}$ is not a minimal generator of $Q.$ Then there exist two linear forms, $\ell$ and $v$ such that 
$\overline{g}=\overline{\ell}\overline{v}$ with $\overline{\ell}$ a minimal generator of $Q.$ By multidegree considerations, we derive 
that $\ell$ is a minimal linear generator of $J_{G_1^\prime,G_2^\prime}:x_{mi}$ where $G_1^\prime$ is the restriction of $G_1$ to 
$\{m-2,m-1,m\}$ and $G_2^\prime$ is the restriction of $G_2$ to $\{i,j,k\}$. But  the ideal quotient 
$J_{G_1^\prime,G_2^\prime}:x_{mi}$ has no minimal linear generator as the following lemma shows. 
\end{proof}

The proof of this lemma uses standard arguments involving Gr\"obner basis theory, but we include all the details for the conveneince of the reader.

\begin{Lemma}
Let $ X=
\left(\begin{array}{ccc}
	x_1 & y_1 & z_1\\
	x_2 & y_2 & z_2\\
	x_3 & y_3 & z_3\\
\end{array}\right)
$ be a $3\times 3$-matrix of indeterminates and let $I\subset K[x_i,y_i,z_i: i=1,2,3]$ be the binomial ideal generated by the the following 
$2$-minors of $X:$ 
\[[12|12], [13|12], [23|12], [12|13], [13|13], [23|13].
\] Then $I:(x_3)=I_2(X)$ where $I_2(X)$ denotes the ideal generated by all $2$-minors of the matrix $X.$
\end{Lemma}

\begin{proof}
A straightforward calculation shows that the reduced Gr\"obner basis of $I$ with respect to the lexicographic order induced by 
$x_1>y_1>z_1>x_2>y_2>z_2>x_3>y_3>z_3$ consists of the generators of $I$ together with the following binomials of degree $3$:
$x_2[12|23], x_3[12|23], x_3[13|23], x_3[23|23].$ In particular, this implies that $\ini_{\lex}(I):(x_3)=\ini_{\lex}(I_2(X))$ since the generators of $I_2(X)$ form the reduced Gr\"obner basis of $I_2(X)$ with respect to the above  lexicographic order.

Clearly, we have $I_2(X)\subseteq I:(x_3).$ Let us assume that there exists a polynomial $f\in I:(x_3)$ such that $f\not\in I_2(X).$
Reducing the polynomial $f$ modulo $I_2(X),$ we may assume that no monomial in the support of $f$ belongs to $\ini_{\lex}(I_2(X)).$ On the other hand, $x_3f\in I,$ thus $x_3\ini_{\lex}f \in I$ which implies that $\ini_{\lex} f\in \ini_{\lex}(I):(x_3)=\ini_{\lex}(I_2(X)),$ contradiction. Therefore, we have $I_2(X)= I:(x_3).$
\end{proof}

\medskip

Of course it is natural to ask whether any Koszul algebra defined by a binomial edge ideal associated to a pair of graphs has a Koszul filtration as it was introduced in \cite{CTV}. Some computer experiments give some hope that the following question may have a positive answer.

\begin{Question}
Let $G_1,G_2$ be two connected graphs on the vertex sets $[m], [n]$ respectively, where $m,n\geq 3,$ and $R=S/J_{G_1,G_2}.$ Assume that $R$ is Koszul. Does it admit a Koszul filtration?
\end{Question}

We end this section by proving a result inspired by \cite[Proposition 2.3]{EHH2}. First, we recall the definition of c-universally Koszul algebras from \cite{EHH2}.

\begin{Definition}{\em 
Let $R$ be a standard graded $K$--algebra. $R$ is called {\em c-universally Koszul} if the set  consisting of all ideals which are generated by subsets of the
variables is a Koszul filtration of $R$.}
\end{Definition}

\begin{Proposition}
Let $G_1,G_2$ be two connected graphs on the vertex sets $[m], [n]$ respectively, where $m,n\geq 3.$ Then $R$ is c-universally Koszul
if and only if $G_1$ and $G_2$ are complete graphs.
\end{Proposition}

\begin{proof}
If $G_1$ and $G_2$ are complete graphs, then $J_{G_1,G_2}$ is the ideal of all the $2$--minors of the matrix $X=(x_{ij})$. This is exactly the defining ideal of the Segre product of the polynomial rings over $K$ in $m$ and, respectively, $n$ indeterminates. By \cite[Proposition 2.3]{HHR},
it follows that $R$ is strongly Koszul, therefore c-universally Koszul. 

For the converse, let us assume that $R$ 
is c-universally Koszul. We have to show that $G_1,G_2$ are complete graphs. Let us assume, for example,  that $G_2$ is not complete. By relabeling 
its vertices if necessary, we may assume that $\{1,2\},\{2,3\}\in E(G_2)$ while $\{1,3\}\notin E(G_2).$ With similar techniques to those used in the proof of Theorem~\ref{Koszul}, we find that $\overline{g},$ where $g=[m-1\ \ m|13],$ is a minimal generator of the ideal quotient 
$0: \overline{x}_{m2}$, hence we get a contradiction to our hypothesis. 

Indeed, let us assume that $\overline{g}$ is not a minimal generator of $0: \overline{x}_{m2}$. Thus, there exist two linear forms, $\ell$ 
and $v$, with $\ell$ a minimal generator of $0: \overline{x}_{m2}$ such that $\overline{g}=\overline{\ell}\overline{v}.$ We have 
$\overline{\ell}\overline{x}_{m2}=0$ if and only if $\ell x_{m2}\in J_{G_1,G_2}.$ As $m-\deg(\ell)\leq m-\deg(g)=
\varepsilon_{m-1}+\varepsilon_m+\varepsilon_{1+m}+\varepsilon_{3+m}, $ we may express $\ell x_{m2}$ as a linear combination of the generators of the ideal $J_{G_1^\prime, G_2^\prime}$ where $G_1^\prime$ consists of the edge $\{m-1,m\}$ of $G_1$ and $G_2^\prime$ 
is the restriction of $G_2$ to $\{1,2,3\},$ thus $G_2^\prime$ consists of two edges, namely $\{1,2\}, \{2,3\}.$ Thus, in the ring 
$S^\prime=K[\{x_{ij}: m-1\leq i \leq m, 1\leq j\leq 3\}],$ $\ell x_{m2}$ belongs to the ideal of $S^\prime$ generated by the minors 
$h_1=[m-1\ \ m|12], h_2=[m-1\ \ m|23].$ Therefore, $\ell$ is a linear generator of $(h_1,h_2):x_{m2}.$ But one easily checks with Singular, or by direct computation, that $(h_1,h_2):x_{m2}=(h_1,h_2,g),$ hence $(h_1,h_2):x_{m2}$ has no linear generator.
 \end{proof}

\section*{Acknowledgment}
We would like to thank the referee for the careful reading of the paper.

{}


\begin{thebibliography}{99}

\bibitem{CDR} A. Conca, E. De Negri, M. E.  Rossi, \emph{Koszul algebras and regularity}, in Commutative Algebra: Expository Papers Dedicated to David Eisenbud on the Occasion of His 65th Birthday, I. Peeva Ed, (2013), 285--315.

\bibitem{CTV} A. Conca, N. V. Trung, G. Valla, \emph{Koszul property for points in projective space}, Math. Scand. \textbf{89} (2001), 201--216

\bibitem{Cru} M. Crupi, \emph{Closed orders and closed graphs},  An. \c Stiin\c t. Univ. "Ovidius" Constan\c ta Ser. Mat. \textbf{24} (2) (2016),
159--167.

\bibitem{CR} M. Crupi, G. Rinaldo, {\em Binomial edge ideals with quadratic Gr\"obner bases}, Electron. J. 
Combin. {\bf 18} (2011), no. 1, \# P211.

\bibitem{DGPS}
W. Decker,  G.-M. Greuel, G.  Pfister, H. Sch{\"o}nemann, 
\newblock {\sc Singular} {4-1-0} --- {A} computer algebra system for polynomial computations.
\newblock {http://www.singular.uni-kl.de} (2016).

\bibitem{D} G. A. Dirac,  \emph{On rigid circuit graphs},   Abh. Math. Sem. Univ. Hamburg
{\bf 38} (1961), 71--76.

\bibitem{EH} V.~Ene, J.~Herzog, \textit{Gr\"{o}bner bases in commutative algebra},  Grad. Stud. Math. {\bf 130}, Amer. Math. Soc.,
Providence, RI 2012.

\bibitem{EHH3}  V.~Ene, J.~Herzog, T.~Hibi,  \textit{Cohen-Macaulay binomial edge ideals}, Nagoya Math. J. {\bf 204} (2011),  57--68.

\bibitem{EHH1} V. Ene, J. Herzog, T. Hibi, \emph{Koszul binomial edge ideals}, Bridging Algebra, Geometry, and Topology, Springer Proceedings in Mathematics \& Statistics, {\bf 96}, D. Ibadula, W. Veys (Eds.) Springer, 2014, 127--138. 

\bibitem{EHH2} V. Ene, J. Herzog, T. Hibi, \emph{Linear flags and Koszul filtrations}, Kyoto J. Math. {\bf 55} no.3 (2015), 517--530.

\bibitem{EHHQ} V. Ene, J. Herzog, T. Hibi, A. Qureshi, \emph{ The binomial edge ideal of a  pair of graphs}, Nagoya Math. J. {\bf 213}
 (2014), 105--125.

\bibitem{Fi} P. C. Fishburn, \emph{Interval graphs and interval order}, Discrete Math. {\bf 55} (1985), 135--149.

\bibitem{Ga} F. Gardi, \emph{The Roberts characterization of proper and unit interval graphs},
Discrete Math. {\bf 307}(22) (2007), 2906--2908.


\bibitem{H} G. Haj\'os, \emph{Uber eine Art von Graphen}, Internat. Math. Nachr. {\bf 11}(1957) page 65.

\bibitem{HHBook}
J.~Herzog, T.~Hibi, \textit{Monomial ideals},  Grad. Texts in Math. {\bf 260} Springer, London, 2010.

\bibitem{HHHKR} J.\ Herzog, T.\ Hibi, F.\ Hreinsd\'ottir, T.\ Kahle, J.\ Rauh, \textit{Binomial edge ideals and conditional independence statements}, Adv. Appl. Math. \textbf{45} (2010) 317--333.

\bibitem{HHR} J. Herzog, T. Hibi, G. Restuccia, \textit{Strongly Koszul algebras}, Math. Scand. {\bf 86}(2000), 161--178.

\bibitem{LO} P. J. Looges, S. Olariu, \emph{Optimal greedy algorithms for indifference graphs}, Comput. Math. Appl. \textbf{25} (1993), 15--25.

\bibitem{OHH} H. Ohsugi, J. Herzog, T. Hibi, \emph{ Combinatorial pure subrings},  Osaka J. Math. {\bf 37} (2000) 745--757.

\bibitem{P} S. B. Priddy, \emph{Koszul resolutions,} Trans. Amer. Math. Soc. {\bf 152} (1970), 39--60.

\bibitem{R} F. S. Roberts, {\em Indifference graphs}, In "Proof Techniques in Graph Theory" (F. Harary, ed.), Academic Press, New York (1969), 139--146.


\bibitem{Sara2} S. Saeedi Madani, D. Kiani, {\em On the binomial edge ideal of a pair of graphs}, Electron. J. 
Combin, {\bf 20} (2013), no. 1, \# P48.




\end{thebibliography}
\end{document}